\documentclass[12pt,oneside,reqno]{amsart}

\usepackage{amsmath, amsthm, amsfonts, amsfonts,amssymb,amscd,amsmath,latexsym,enumerate,verbatim, calc, quotmark, mathtools, txfonts, calrsfs, imakeidx, fancyhdr, listings, lipsum, tkz-euclide, amsthm}

\usepackage[pagewise]{lineno}
\usepackage{mathtools}

\usepackage[T1]{fontenc}
\usepackage[utf8]{inputenc}

\usepackage[square,sort,comma,numbers]{natbib}
\usepackage[all, cmtip]{xy}

\usepackage[pagebackref]{hyperref}
\hypersetup{
colorlinks = true, linkcolor = red,
citecolor   = blue,
urlcolor    = blue,
}

\newtheorem{innercustomgeneric}{\customgenericname}
\providecommand{\customgenericname}{}
\newcommand{\newcustomtheorem}[2]{%
\newenvironment{#1}[1]
{%
	\renewcommand\customgenericname{#2}%
	\renewcommand\theinnercustomgeneric{##1}%
	\innercustomgeneric
}
{\endinnercustomgeneric}
}

\newcustomtheorem{customthm}{Theorem}
\newcustomtheorem{customqn}{Question}

\def\NZQ{\mathbb}               
\def\NN{{\NZQ N}}

\def\ZZ{{\NZQ Z}}

\def\n{\mathfrak{n}}

\def\q{\mathfrak{q}}
\newcommand{\N}{\mathbb{N}}
\newcommand{\Z}{\mathbb{Z}}

\newcommand{\p}{\mathfrak{p}}
\newcommand{\m}{\mathfrak{m}}
\DeclareMathOperator*{\Um}{Um}
\DeclareMathOperator*{\ET}{ETrans}
\DeclareMathOperator*{\h}{H}

\def\u{\mathbf{u}}
\def\v{\mathbf{v}}
\def\w{\mathbf{w}}

\def\GL{{\operatorname{GL}}}          
\def\SL{{\operatorname{SL}}}          
\def\vpmod{\!\!\pmod} 
\DeclareMathOperator*{\Ht}{ht} 

\DeclareMathOperator*{\Hom}{Hom}

\DeclareMathOperator*{\Aut}{Aut}

\DeclareMathOperator{\E}{E}

\DeclareMathOperator*{\spec}{Spec}
\DeclareMathOperator*{\Pic}{Pic}

\newtheorem{theorem}{Theorem}[section]
\newtheorem{lemma}[theorem]{Lemma}
\newtheorem{corollary}[theorem]{Corollary}
\newtheorem{proposition}[theorem]{Proposition}

\newtheorem{question}[theorem]{Question}

\newtheorem{example}[theorem]{Example}

\newtheorem{definition}[theorem]{Definition}
\newtheorem{Fact}[theorem]{Fact}

\newtheorem*{acknowledgement}{Acknowledgement}

\title{On Projective modules over graded $R$-subalgebras of $R[X,1/X]$}
\author{Diksha Garg, Anjan Gupta}

\address{Department of Mathematics\\
Indian Institute of Science Education and Research Bhopal\\
Bhopal Bypass Road, Bhopal, Madhya Pradesh, India. Pin - 462066.}
\email{diksha19@iiserb.ac.in, anjan@iiserb.ac.in}
\email{}

\thanks{Corresponding author: Diksha Garg; 
{\it email: diksha19@iiserb.ac.in}}

\begin{document}
\begin{abstract} 
Let $R$ be a Noetherian ring of dimension $d$ and $A$ be a graded $R$-subalgebra of $R[X,1/X]$. Let $P$ be a projective module over $A$ of rank $r \geq \max\{d+1,2\}$ and $\v=(a,p)$ be a unimodular element of $A \oplus P$. We find an elementary automorphism $\tau$ such that $\tau (\v)  = (1, 0)$. Consequently, we obtain the cancellative property of $P$. We show that $P$  splits off a free summand of rank one. When $A = R[X]$ or $R[X, 1/ X]$, the results are well-known due to the contributions by various authors. 
\end{abstract}
\maketitle 


\section{introduction}

Assume that $R$ is a commutative Noetherian ring of dimension $d$ and $P$ is a finitely generated projective module over $R$. A classical result due to Serre \cite{serre1958modules} asserts that if the rank of $P$ is at least $d + 1$, then $P$ splits off a free summand of rank one. A subsequent result of Bass \cite{BassH} states that under the same condition on rank, the projective module $P$ is cancellative, i.e. $P \oplus R^n \cong Q \oplus R^n$ implies $P \cong Q$. Research on this subject gathered momentum and much optimism after the seminal work of Quillen  \cite{quillen1976projective} and Suslin \cite{suslin1977structure} who independently proved that any projective module over a polynomial algebra $k[X_1, \ldots, X_n]$ over a field $k$  is free. The innovative method of their proof inspired several authors to study projective modules over different types of algebras over $R$. We record a few to draw inspiration. 
Bhatwadekar et al.  \cite{bhatwadekar1985bass} proved that any finitely generated
projective module over $R[X_1, \ldots, X_n ,Y_1^{\pm 1}, \ldots, Y_m^{\pm 1}]$ of rank $r \geq  d + 1$ splits off a free summand of rank one. Later Rao \cite{rao1988question} proved that any projective module over $R[X_1, \ldots, X_n]$  of rank $r \geq d + 1$ is cancellative. Lindel \cite{lindel1995unimodular} extended Rao’s result by showing that any projective module over $R[X_1, \ldots, X_n, Y_1^{\pm 1}, \ldots, Y_m^{\pm 1}]$  of rank $r \geq d + 1$ is cancellative.

We refer the reader to \S 2 for definitions of the set $\Um_n(R)$ of unimodular rows, the group $\E_n(R)$ of elementary matrices, the group $\ET(R \oplus P)$ of elementary transvections and for other notations and terminology. 
Recently in 2018, Gubeladze proved the following result \cite[Theorem 1.1]{gubeladze2018unimodular}.
\begin{theorem}{\rm(Gubeladze)}
Let $R$ be a Noetherian ring of dimension $d$ and $M$ be a commutative cancellative monoid. Then $\Um_{n}(R[M])=e_1\E_n(R[M])$ for all $n\geq \max\{d+2,3\}$.
\end{theorem}

For a partially cancellative, seminormal, torsion free monoid $M$ and  a projective $R[M]$-module $P$ of rank $r > \max \{1, d\}$,  Sarwar \cite[Theorem 3.4]{sarwar2022k0} proved  that $\Um(R[M] \oplus P)=e_1 \ET(R[M] \oplus P)$.

Motivated by these theorems and taking a cue from the earlier results, we pose the following natural questions \ref{Q1}, \ref{Q2}, \ref{Q3} for a commutative Noetherian ring $R$ of dimension $d$ and a graded $R$-subalgebra $A$ of the Laurent polynomial algebra $R[X_1^{\pm 1}, \ldots, X_m^{\pm 1}], \deg(X_i) = 1$. 
\begin{question}\label{Q1}
Is it true that $\Um_n(A) = e_1 \E_n(A)$ for all $n\geq \max \{d+2, 3\}$?
\end{question}

More generally we may ask the following : 
\begin{question}\label{Q2}
Let $P$ be a projective $A$-module of rank  $r \geq \max \{2, d + 1\}$. Is it true that $\Um(A \oplus P) = e_1 \ET(A \oplus P)$?
\end{question}
An affirmative answer to  Question \ref{Q2} shows that $P$ is cancellative. The next question concerns the splitting of projective modules. 

\begin{question}\label{Q3}
Let $P$ be a projective $A$-module  of trivial determinant of rank at least $d + 1$. Does $P$ split off a free summand of rank one?
\end{question}

%
%

For an affirmative answer to Question \ref{Q3}, the assumption on the triviality of determinant cannot be dropped since 
$\Pic k[X^2, X^3] = (k, +)$.
An example due to Bhatwadekar et al.  shows that answer to Question \ref{Q3} may be negative without the assumptions that $A$ is graded or that $A$ is a subalgebra of $k[X_1^{\pm 1}, \ldots, X_m^{\pm1}]$ (see Example \ref{bhatex}).
An assiduous reader may observe that the subalgebras $A$ in the questions above include a large classes of well-known algebras viz. polynomial algebras $R[X_1, \ldots, X_k]$, monoid algebras $R[M], M \subset \ZZ^m$ and subalgebras of $R[t, 1/t]$ like Rees algebras $R[It]$, extended Rees algebras $R[It, \frac{1}{t}]$, symbolic Rees algebras $R[I^{(n)}t^n : n \geq1]$ corresponding to an ideal $I$ of $R$. Special care is required since $A$ may not be Noetherian. We refer the reader to  \cite{roberts1985prime} for examples of non-Noetherian symbolic Rees algebras.

Finding complete answers to the above questions appears to be a Herculean task at present. Complete answers to Questions \ref{Q2}, \ref{Q3} are unknown even in the special case when $A = R[M]$ is a monoid algebra for some submonoid $M$ of $\NN^m$. 
Nevertheless,  recent results of Rao, Sarwar \cite{rao2019stability} and Banerjee \cite{banerjee2024subrings} shed some light on these questions in special cases. 
If $R$ is a domain and $A \subset R[t, 1/ t]$ is a Rees algebra or an extented Rees algebra, then Question \ref{Q3} has an affirmative answer \cite[Theorem 3.5]{rao2019stability}.
Later Banerjee \cite[Theorems 3.3, 4.2, 4.5]{banerjee2024subrings} proved that any projective module of rank $d + 1$ over a Noetherian $R$-subalgebra $A$ of $R[X, 1/ X]$ splits off a free summand of rank one when $A_s = R_s[X]$ and is cancellative when $A_s = R_s[X]$ or $R_s[X, 1/ X]$ for some nonzero divisor  $s$  of $R$. The main objective of this article is to settle the above questions when the number of variables $m$ is equal to $1$. We do not assume that the subalgebra $A$ is Noetherian in our results. We state our results below.

\begin{theorem}\label{main1}{\rm(Theorem \ref{main})}
Let $R$ be a Noetherian ring {of dimension $d$}, $A=\oplus_{n\in \Z}A_{n}t^{n}$  be a $\Z$-graded $R$-subalgebra of $R[t,1/t]$ and $I$ be a homogeneous ideal {of $A$}. Then
$${\Um}_n (A, I)=e_1\E_n(A, I) \ \text{for all n} \ \geq \max\{3,d+2\}.$$
\end{theorem}

The above theorem is new in the very special case when $A= R[t]$ and $I = (t^m), m \geq 2$. To the best of our knowledge, a unimodular row $\v \in \Um_{d + 2}(R[t], (t^m))$, $d = \dim R$ is known to be completable to a relative elementary matrix in $\E_{d + 2}(R[t], (t^m))$ only when $m = 1$, i.e. $\frac{R[t]}{(t)} = R$ is a retract of $R[t]$.

\begin{theorem}{\label{main2}}{\rm (Theorem \ref{43})}
Let  $R$ be a Noetherian ring {of dimension d}, $A$ be a graded R-subalgebra of $R[t,1/t]$ and $I$ be a homogeneous ideal of $A$. Let $P$ be a finitely generated projective $A$-module of rank  $r \geq \max\{d+1,2\}$. Then $\Um(A \oplus P, I)= e_1\ET(A \oplus P, I)$.
\end{theorem}

As an application of the above two results, we establish the following theorem which ensures the existence of a unimodular element in a projective $A$-module $P$ of rank at least $d+1$.

\begin{theorem}\label{main3}{\rm(Theorem \ref{uni3})}
Let $R$ be a Noetherian ring of dimension $d$ and $A$ be a graded $R$-subalgebra of $R[t, 1/ t]$. Let $P$ be a projective $A$-module of rank $r \geq d + 1$ whose determinant is extended from $R$.  
Let $\phi: A \rightarrow R$ be an $R$-algebra homomorphism and $I= Ker(\phi)$. Then the natural map $\pi: \Um(P) \rightarrow \Um(P/IP)$ is surjective.
\end{theorem}

Now we demonstrate the layout of the paper.  In Section 2, we review basic definitions and important results which we require to prove our theorems. 
In Section 3, we prove Theorem \ref{main1}.  We use a series of lemmas which mainly fall into three key steps. In the first step, we restrict our attention to the case when  $A \subset R[t]$. We appeal to the local-global principle of positively graded rings to assume that $R$ is local. The proof then follows by a standard argument using generalized dimension functions, thanks to the work of Plumstead \cite{plumstead1983conjectures}. In the next step, we consider the case when $R$ is a domain and $A$ is a graded $R$-subalgebra of $R[t, 1/t]$. The main idea is to apply a descend lemma, see Lemma \ref{sb} which shows that unimodular rows over graded $R$-subalgebras of $R[t, 1/t]$ by elementary row operations can be reduced to unimodular rows over graded $R$-subalgebras of $R[t]$. In the final step, the theorem is proved by induction on the number of minimal primes assuming $R$ is reduced.

In section 4, we establish the proof of   Theorem \ref{main2} and Theorem \ref{main3}. A crucial observation here is  Lemma \ref{L31}. As illustrated in the free case, we corroborate the proof of Theorem \ref{main2} when $R$ is a domain by exploiting the technique introduced by Lindel. The key new idea is implementing Lindel's method by localizing at a homogeneous element of positive degree to make the localization of the projective module free. The method works because our Theorem \ref{main1} is valid for any homogeneous ideal $I$.
The proof of Theorem \ref{main2} is then accomplished by induction on the number of minimal primes assuming $R$ is reduced. As a corollary we get cancellation of projective modules, see Corollary \ref{cancel}. 
In the later part of the section, we prove Theorem \ref{main3}. We invoke Quillen's idea of splitting matrices isotopic to the identity matrix and patching diagram to prove the theorem. As a consequence, we authenticate the existence of unimodular elements of $P$, see Corollary \ref{eum}.

Our method of proof falls short of being used to answer questions \ref{Q1}, \ref{Q2} and \ref{Q3} in full generality. We believe that the present article will garner enough interest from the larger mathematical community to work on these problems. To facilitate access to our research, we presented our article lucidly. The book by T.-Y. Lam \cite{lam2006serre} may be used as a good reference to learn prerequisites.

\section{Preliminaries}

In this section, we establish notation and state important results which we use in the later sections. Let $R$ denote a commutative ring with  $1$. A row vector $\mathbf{v}=\left(v_1, v_2, \ldots, v_n\right) \in R^n$ is said to be unimodular if there exists a row vector $\mathbf{w}=\left(w_1, w_2, \ldots, w_n\right) \in R^n$ such that $\mathbf{v} \cdot \mathbf{w}^t=\sum_{i=1}^n v_i w_i=1$. The collection of all unimodular rows in $R^n$ is denoted by $\Um_n(R)$. We denote by $e_i$  the $ i $-th row of the identity matrix $I_n$ of size $n$.  If  $I$ is an ideal of $R$, then  $\Um_n(R, I)$ denotes the set of all unimodular rows of length $n$ congruent to $e_1$ modulo $I$. One clearly observes that $\Um_n(R, R) = \Um_n(R)$. 

Let $P$ be a finitely generated projective $R$-module. An element $p\in P$ is said to be a unimodular element of $P$ if there exists a $\phi \in \Hom(P, R)$ such that $\phi(p)= 1$. The set of unimodular elements of $P$ is denoted by $\Um(P)$. The automorphisms of $R\oplus P$ of the form $(a, p) \xrightarrow{\tau_q} (a, p+a q)$ and $(a, p) \xrightarrow{\tau_\psi} (a+\psi(p), p)$ for $q \in P$, $\psi \in \Hom(P,R)$ are called elementary transvections of $R \oplus P$. The group generated by such automorphisms is denoted by $\ET(R \oplus P)$. Let $I$ be an ideal of $R$. 
The group generated by $\tau_q, \tau_\psi$  where $q \in IP$ and $\psi \in IP^*$ is denoted by $\ET(I(R \oplus P))$. 
The normal closure of $\ET(I(R \oplus P))$ in $\ET(R\oplus P)$ is denoted by $\ET(R \oplus P, I)$

c The group $\E_n(R)$ is generated by elementary matrices $e_{ij}(\lambda)$ for $ \lambda \in R, i \neq j, 1 \leq i, j \leq n$. Here $e_{ij}(\lambda)$ is a square matrix of size $n$ which differs from the identity matrix $I_n$ only at the $(i, j)$-th place and whose $(i, j)$-th entry is equal to $\lambda$. If $I$ is an ideal of $R$, then $\E_n(I)$ is defined as the subgroup of $\E_n(R)$ generated by elementary matrices $e_{ij}(\lambda)$ for  $i \neq j, 1 \leq i, j \leq n$ and $\lambda \in I$. The normal closure of $\E_n(I)$ in $\E_n(R)$ is denoted by $\E_n(R, I)$. We set $\GL_n(R, I) = \{\alpha \in \GL_n(R) : \alpha \equiv I_n \vpmod I\}$, $\SL_n(R, I) = \GL_n(R, I) \cap  \SL_n(R)$. Clearly, $\E_n(R, I) \subset \SL_n(R, I)$.
For all other unexplained notation and terminology, we refer the reader to \cite{lam2006serre}.

The following result is due to Bass, cf. 
\cite[Page 183, Theorem 3.4]{BassH}. 
\begin{theorem}\label{bass}
Let $R$ be a commutative Noetherian ring of dimension $d$ and $I$ be an ideal of $R$. 
	Let $P$ be a projective module of rank $r\geq\max\{d+1,2 \}$. Then 
	$\Um(R \oplus P, I)=e_1 \ET(R \oplus P, I).$
	In particular, $\Um_n(R, I)=e_1 \E_n(R, I)$ for $n \geq\max\{d+2, 3\}$.
\end{theorem}

The next result is due to Gubeladze, cf. \cite[Corrolary 7.4]{gubeladze1992elementary}.

\begin{lemma}{\rm (Local-global principle)}\label{lg}
	Let $A = A_0 \oplus A_1\oplus A_2 \oplus \cdots~$  be a graded ring and $n\geq3$. Let $\v \in \Um_{n}(A)$ such that the natural image of $\v$ in $\Um_n(A_0)$ is $e_1$. If $\v_{\m} \sim_{\E_n(A_{\m})} (e_1)_{\m}$ for all $\m \in \max(A_0)$, then $\v \sim_{\E_n(A)} e_1$.
\end{lemma}

%

The following result due to van der Kallen, cf. \cite[Lemma 3.22]{van1983group} shows that the row vector obtained by adding a constant multiple of the $i$-th coordinate, $i \geq 2$ to any other coordinate can also be obtained by a sequence of relative elementary operations. 
\begin{lemma}\label{VK}
	Let $I$ be an ideal of a ring $R$ and $\v \in \Um_n(R, I), n \geq 2$. If $i \geq 2$, then for any $r \in R$
	\[\v \sim_{\E_n(R, I)} \v e_{ij}(r) .\]
\end{lemma}

The following result is well-known. In the absence of an appropriate reference, we include a proof for completeness.
\begin{lemma}\label{pnil}
	Let $P$ be a projective $R$-module of rank $r \geq 2$ and $p \in \Um(P)$. Assume that overline denotes the quotient modulo $Nil(R)$. Then we have $\Um(R \oplus P, I)= e_1 \ET(R \oplus P, I)$  if and only if $\Um({\overline{R}} \oplus \overline{P},\overline{I})= e_1\ET(\overline{R}\oplus \overline{P}, \overline {I})$.
\end{lemma}

\begin{proof}
	The "only if" part is straightforward. For the "if" part, we choose $(a,p)\in \Um(R\oplus P,I)$ such that $(\overline{a},\overline{p}) \sim_{\ET(\overline R \oplus \overline P, \overline I)} (1,0)$. By lifting elementary operations, we get 
	$(a,p) \sim_{\ET(R \oplus P, I )}(1+n, p')$ where $n \in J$ and $p' \in JP$ for $J = I\cap Nil(R)$. Clearly $(1+n, p') \sim_{\ET(R \oplus P,  I)} (1+n,0)$ since $(1 + n)$ is a unit. It remains to show that $(1+n,0)\sim_{\ET(R \oplus P, I)}(1,0)$.

	Let $v(T)=(1+nT,0)\in \Um(R[T]\oplus P[T], JR[T])$. We show that $v(T) \sim_{\ET(R \oplus P, I)} v(0)$. Due to the local-global principle for transvection groups, cf. \cite[Theorem 4.8]{chattopadhyay2015equality}, it is enough to assume that $R$ is local and $P$ is free. In this case $v(T)\in \Um_{r + 1}(R[T], JR[T])$. We reduce $v(T)$ to $v(0)$ by a sequence of elementary operations as below. (We set forth the case when $r = 2$ for clarity)
	\begin{align*}
		v(T) \xrightarrow{C_2' = C_2 + nT(1 + nT)^{-1}C_1} (1 + nT, nT, 0) \xrightarrow{C_1' = C_1 - C_2} (1, nT, 0) \xrightarrow{C_2' = C_2 - nTC_1} (1, 0, 0) = v(0)
	\end{align*}
	All operations above correspond to an $\E_3(R[T], JR[T])$-action. The correspondence for the
	middle operation follows from Lemma \ref{VK}. Therefore $v(T) \sim_{\ET(R \oplus P, I)} v(0)$. 
	
	Putting $T = 1$, we obtain $v(1) \sim_{\ET(R \oplus P, J )} (1, 0)$ and therefore the proof follows.
\end{proof}

The next result is well-known. One can prove this by using the fact that every finitely generated projective module is isomorphic to the image of an idempotent endomorphism on a free module of finite rank.

\begin{lemma}\label{pidem}
	Let $A$ be a graded $R$-algebra, $P$ be a finitely generated projective module over $A$ of rank r and $p \in P$. 
	Then there exists a finitely generated graded $R$-subalgebra $A' \subset A$, a finitely generated projective module $P'$ over $A'$ of rank $r$ and an element $p' \in P'$ such that $P = P' \otimes_{A'} A$ and $p= 1 \otimes_{A'}p'$. If $\p \in \Um(P)$, we may further assume that  $p' \in \Um(P')$.
\end{lemma}

The lemma below was proved by Bhatwadekar - Roy, cf. \cite[Lemma 4.1]{bhatwadekar1982stability}. 
It is crucial to the proof of Lemma \ref{sb}. 

\begin{lemma}\label{pbhat1}
	Let $B \subset C$ be rings of dimension $\delta$ and let $x$ be an element of $B$ such that $B_x=C_x$. Then,
	\begin{enumerate}
		\item $B /(1+x b)=C /(1+x b)$ for all $b \in B$,
		\item If $\mathfrak{A}$ is an ideal of $C$ such that  $\Ht \mathfrak{A} \geqq \delta$ and $\mathfrak{A}+xC=C$, then there exists an element $b \in B$ such that $1+x b \in \mathfrak{A}$. 
		%
	\end{enumerate}
\end{lemma}

The result below is due to Lindel, cf. (\cite{lindel1995unimodular}, Lemma 1.1).
\begin{lemma}\label{P1}
	Let $P$ be a projective $R$-module of rank $r$. Assume that $s$ is a non-nilpotent element of $R$ such that $P_s$ is free. Then there exists $p _1, p_2, \ldots, p_r \in P, \phi_1, \phi_2, \ldots, \phi_r \in P^* = \Hom(P, R)$ and $t \in \NN$ such that 
	\begin{itemize}
		\item[1.] $(\phi_i(p_j))_{r \times r} = {diag} (s^t, s^t, \ldots, s^t)$.
		
		\item[2.]  $s^tP \subset F$ and $s^tP^* \subset G$ with $F = \sum_{i = 1}^r Rp_i$ and $G = \sum _{i = 1}^r R\phi_i$.
	\end{itemize}
\end{lemma}

{\bf Throughout the rest of the paper $R$ denotes a commutative Noetherian ring with $1$ and $A$ denotes a graded $R$-subalgebra of $R[t, 1/t]$. Projective modules are assumed to be finitely generated. Ring homomorphisms are assumed to preserve unity.}
\section  {Elementary completion of Unimodular rows over Graded rings}

In this section, we show that a unimodular row over a graded $R$-subalgebra $A$ of $R[t,1/t]$ can be completed to an elementary matrix. If $A = R[t]$ or $A = R[t, 1/t]$, the result is due to Suslin, cf. \cite[Theorem 7.2]{suslin1977structure}. We begin by proving a lemma which concerns the special case $A \subset R[t]$. The proof follows by using the local-global principle and the concept of generalized dimension functions. 

\begin{lemma}\label{T31}
	Let $R$ be a ring of dimension $d$ and  $A$ be a graded $R$-subalgebra of $R[t]$. Then $\Um_n(A) = e_1\E_n(A)$ for all $ n\geq \max \{3,d+2\}$. 
\end{lemma}

\begin{proof}
	We begin with the observation that $R/Nil(R) \hookrightarrow A/Nil(A)  \hookrightarrow  R/Nil(R)[t]$. We may assume that $R$ is reduced, see Lemma \ref{pnil}. Let $\v,\w \in \Um_n(A)$ such that $\v\w^t=1$. We show that $\v \in e_1\E_n(A)$. Replacing $A$ by the $R$-algebra generated by the homogeneous components of coordinates of $\v$ and $\w$, we may further assume that $A$ is Noetherian. Applying dimension inequality, cf. \cite[Theorem 15.5]{matsumura1989commutative} on different extensions $R / \p \hookrightarrow A/ A \cap \p R[t]$ where $\p$ is a minimal prime of $R$, we obtain  $\dim A \leq d + 1$.

 If $\dim R = 0$, then $\dim A \leq 1$. It follows by Theorem \ref{bass} that $\v \sim_{\E_n(A)} e_1$.  
	Now we consider the case when $\dim R \geq 1$. Let $A_{+} = \oplus_{i \geq 1} A_i$ denote the irrelevant ideal of $A$. Since $A/ A_{+} = R$ and $n \geq \dim R  + 2$, the image $\overline{\v} \in \Um_n(R)$ of $\v$ modulo $A_{+}$ can be reduced to $e_1$ by applying elementary operations. Lifting these operations, we obtain $\v \sim_{\E_n(A)} \v'$ for some $\v' \in \E_n(A, A_{+})$.
	
	We need to show that  $\v' \sim_{\E_n(A)} e_1$. 
	Due to the local-global principle of graded rings, see Lemma \ref{lg},  it is enough to assume that $R$ is a local ring. Since $R$ is reduced and  $\dim R  \geq 1$, the maximal ideal contains a nonzero divisor $s$. We observe that  $\spec A =V(s)\sqcup D(s)$ where each of $V(s) = \{\p \in \spec A : s \in \p\}$ and $D(s) =  \{\p \in \spec A : s \not\in \p\}$ has dimension at most $d$, so $\spec A $ admits a generalized dimension function $\partial : \spec A \rightarrow \NN$ such that $\partial(\p) \leq d$ for all $\p \in \spec A$. Therefore, $\v' \sim_{\E_n(A)} e_1$ by \cite[Theorem A]{eisenbud1973generating}, \cite[p. 1420]{plumstead1983conjectures} and the result follows.
\end{proof}

\begin{lemma}\label{t31a}
	Let $R$ be a ring of dimension $d$,  $A$ be a graded $R$-subalgebra of $R[t]$ and $I = \oplus_{i \geq 1} I_i$ be a  homogeneous ideal of $A$. Then $\Um_n(A, I) = e_1\E_n(A, I)$ for all $ n\geq max \{3,d+2\}$. 
\end{lemma}
\begin{proof}
	Let $\v \in \Um_n(A, I)$. We get a $\w \in  \Um_n(A, I)$ such that $\v \w^t = 1$ by a result of Vaserstein, cf. \cite[Lemma 2]{vaserstein1969stabilization}. We consider the graded $R$-subalgebra $A' = R \oplus (\oplus_{i \geq 1} I_i) \subset R[t]$. We observe that $I \subset A'$ and both $\v, \w$ are row vectors in  $A'^n$, therefore $\v \in \Um_n(A', I)$. 
	
	By Lemma \ref{T31} there exists $\alpha \in \E_{n}(A')$ such that $\v = e_1 \alpha$. Comparing the images of both sides modulo $I$, we get $e_1 = e_1\overline{\alpha}$. Since $R = A' / I$ is a retract of $A'$, we have $\overline{\alpha} \in \E_{n}(R) \subset \E_{n}(A')$. Replacing $\alpha$ by $\overline{\alpha}^{-1}\alpha$, we obtain $\v = e_1 \alpha$ for $\alpha \in  \E_{n}(A') \cap  \SL_{n}(A', I)$. Since $R = A' / I$ is a retract of $A'$, we have $\E_{n}(A') \cap  \SL_{n}(A', I) = \E_{n}(A', I)$, cf. \cite[p. 204, Lemma 1.5]{lam2006serre}, so $\v \sim_{\E_n(A', I)} e_1$.
	The proof is completed here 
	since $\E_{n}(A', I) \subset \E_{n}(A, I)$.
\end{proof}

The following is a wide generalization of a result of Banerjee, cf. \cite[Lemma 4.4]{banerjee2024subrings}. It plays a vital role in the proof of Lemma \ref{mainid}.

\begin{lemma}\label{sb}
	Let $R \subset S$ be rings of dimension at most $d \geq 2$ and $x \in R$ be a nonzero divisor in $S$ such that $R_x = S_x$. Then the canonical map 
	\[\frac{\Um_{d + 1}(R, x^n)}{\E_{d + 1}(R, x^n)} \rightarrow \frac{\Um_{d + 1}(S, x^n)}{\E_{d + 1}(S, x^n)}\]
	is surjective for all $n \geq 0$.
\end{lemma}

\begin{proof}
	Let $\v = (v_1, \ldots, v_{d+1}) \in \Um_{d + 1} (S, x^n)$. We need to find $\alpha \in \E_{d + 1} (S, x^n)$ such that $$\v \alpha \in   {\Um}_{d + 1} (R, x^n).$$
	The proof falls into two cases.\\
	\noindent
	{\bf Case - 1} : ($n \geq 1$)\\
	Let $y = x^n \in R$ and $v_1 = 1 - sy$ for some $s\in S$. Note that $(v_1, \ldots, v_{d }, yv_{d+1}) \in \Um_{d + 1} (S, y)$. We choose $s_1, \ldots, s_{d} \in S$ such that $\Ht (v_1 + s_1yv_{d+1}, \ldots, v_{d} + s_{d}yv_{d+1}) \geq d$, cf. \cite[Chapter III, Lemma 3.4]{lam2006serre}.
	Therefore, right multiplying $\v$ by elementary matrices of type $e_{{d +1}\,j}(ys_j) \in \E_{d + 1}(S, y)$, $j = 1, \ldots, d$ we may assume that $ht(v_1, \ldots, v_d) \geq d$. Since $\dim S  \leq d$, two possibilities arise. If $\Ht (v_1, \ldots, v_d) \geq d + 1$, then $(v_1, \ldots, v_d)\in \Um_d(S, y)$. Consequently  $\v \alpha = e_1$ for some $\alpha \in \E_{d + 1} (S, y)$ and the result follows. 
	
	Therefore we assume that $ht(v_1, \ldots, v_d) = d$ which forces $\dim S = d$.
	Let $\n$ be a minimal prime ideal over $(v_1, \ldots, v_d)$. Then $ht(\n) = d$. Let $\m = \n \cap R$. Since $v_1 = 1 - sy \in \n$, we have $y \not\in \n$ and consequently $y \not\in \m$.
	
	By our given hypothesis, we have $R_y = S_y$ which implies that for any $s \in S$ there exists $n \in \NN$ such that $y^ns \in R$. We have the inclusion of local rings $R_{\m} \hookrightarrow  S_{\n}$. If $\frac{u}{v} \in S_{\n}$, then we may choose $n \in \NN$ large enough such that $y^nu, y^n v \in R$ and so $\frac{u}{v}  = \frac{y^nu}{y^nv} \in R_{\m}$. Therefore this inclusion is an isomorphism. It follows that $\Ht \m  = \dim R_{\m} = \dim S_{\n} = \Ht \n = d$. 
	Since by our hypothesis $\dim R \leq d$, we must have $\dim R = d$. The ideal $(v_1, \ldots, v_d)$ is comaximal to $y$. By Lemma \ref{pbhat1}, we obtain $r \in R$ such that 
	\[1 - ry \in (v_1, \ldots, v_d).\]
	
	Let $v_1 = 1 - t_1y, v_2 = t_2y, \ldots, v_{d + 1} = t_{d + 1}y$ for some $t_1, \ldots t_{d + 1} \in S$.
	We choose $n \in \NN$ large enough such that $y^nt_i \in R$ for $i = 1, \ldots, d + 1$. We observe that $y(1 - ry) -  v_{d + 1}(1 - r^{n + 1}y^{n + 1})$ is divisible by $y(1 - ry)$. 
	Furthermore,
	$$y(1 - ry) - v_{d + 1}(1 - r^{n + 1}y^{n + 1}) = y( 1 - ry) - v_{d + 1}  + r^{n + 1}y^{n + 2}t_{d +1}= y - y^2(r - r^{ n + 1}y^n t_{d + 1}) - v_{d + 1} $$
	
	If we set $z = r - r^{ n + 1}y^n t_{d + 1}$, then $z \in R$ and $y - y^2z - v_{d + 1} \in (yv_1, \ldots, yv_d)$.
	Right multiplying $\v$ by elementary matrices of type $e_{i \ d + 1}(ys), i \geq 2, s \in S$, we obtain
	$$\v \sim_{\E_{d + 1}(S, y)} (v_1, v_2, \ldots, v_d, y - y^2z ) =(1 - t_1y, t_2y, \ldots, t_dy, y - y^2z ).$$
	%
	%
	%
	%
	%
	%
	One observes that $t_jy - t_jy^{n+1}z^n = yt_j(1 - y^nz^n)$ is divisible by $y - y^2z$. Therefore, by column operations of type  $C_j' = C_j + sC_{d + 1}, s \in S, j = 1, \ldots ,d$, we may reduce $(1 - t_1y, t_2y, \ldots, t_dy, y - y^2z )$ to
	$(1 - t_1y^{n + 1}z^n, t_2y^{n + 1}z^n, \ldots, t_dy^{n + 1}z^n, y - y^2z )$. By Lemma \ref{VK}, we obtain
	\[ (1 - t_1y, t_2y, \ldots, t_dy, y - y^2z ) \sim_{\E_{d + 1}(S, y)} (1 - t_1y^{n + 1}z^n, t_2y^{n + 1}z^n, \ldots, t_dy^{n + 1}z^n, y - y^2z ).\]
	The later row is in $ {\Um}_{d + 1} (R, y)$ since $y, z \in R$ and also $y^nt_i \in R$ for $i = 1, \ldots, n$. Hence the proof follows in this case.
	
	\noindent 
	{\bf Case - 2} : ($n = 0$)\\ 
	Here $\v \in \Um_{d + 1} (S)$. Since $x \in S$ is a nonzero divisor, we have $\dim S/ (x) \leq d - 1$. By Theorem \ref{bass}, the image $\overline{\v} \in \Um_{d + 1} (S/ (x))$ can be reduced to $e_1$ by elementary operations. By lifting these elementary operations, we obtain a $\alpha \in \E_{d + 1}(S)$ such that $\v \alpha \in  \Um_{d + 1} (S, (x))$. By Case - 1, we have a $\beta \in \E_{d + 1}(S, (x))$ such that $\v \alpha \beta \in \Um_{d + 1} (R, (x)) \subset \Um_{d + 1} (R)$. Therefore, the proof is completed here. 
\end{proof}

The next lemma serves a pivotal role in the proof of our main theorem.

\begin{lemma}\label{mainid}
	Let $R$ be an integral domain of dimension $d$, $A=\oplus_{n\in \Z}A_{n}t^{n}$  be a $\Z$-graded $R$-subalgebra of $R[t,1/t]$ and $I$ be an ideal of $A$ containing a homogeneous element. Then
	$${\Um}_n (A, I)=e_1\E_n(A, I)\quad \ \text{for all n} \ \geq \max\{3,d+2\}.$$
\end{lemma}

\begin{proof}
	
	As argued in Lemma \ref{T31}, we may assume that $A$ is Noetherian. 
	We may further assume that $R \subsetneq A$, $I \neq 0$ and $n = d + 2$  since otherwise, the result is obvious by Theorem \ref{bass}. 
	We choose a homogeneous element $x$ in $I$ of nonzero degree. Without loss of generality, we may assume that $\deg(x) > 0$. Let $A_{+} = \oplus_{n \geq 0}A_{n}t^{n}$. One  clearly observes that $A[1/ x] = A_{+}[1/ x]$. We also have that both $\dim A $ and $\dim A_{+} $ are at most $d + 1$.
	
	Let $\v \in {\Um}_{d+ 2} (A, I)$ and $\overline{\v} \in {\Um}_{d+2} (\overline{A},\overline{I})$ denote the image of $\v$ modulo $(x)$. Since $\dim A/ (x) \leq d$, we have $\overline{\v} \sim_{\E_{d+2} (\overline{A},\overline{I})} e_1$ by Theorem \ref{bass}. 
	Lifting elementary operations, we obtain $\v \sim_{\E_{d+2} (A, I)} \v'$ for some $\v' \in \Um_{d+ 2}(A, (x))$. By Lemma \ref{sb}, we have $\v' \sim _{\E_{d+ 2} (A, (x))} \v''$ for some $\v'' \in {\Um}_{d+2} (A_{+}, (x))$. The result follows since $\v'' \sim _{\E_{d+2} (A_{+}, (x))} e_1$ by Lemma \ref{t31a} and both $\E_{d+2} (A_{+}, (x))$, $\E_{d+2}(A,(x))$ are subset of $E_{d+ 2} (A, I)$.
\end{proof}

{Now we are in a position to prove the main result of this section.}
\begin{theorem}\label{main}
	Let $R$ be a ring {of dimension $d$}, $A=\oplus_{n\in \Z}A_{n}t^{n}$  be a $\Z$-graded $R$-subalgebra of $R[t,1/t]$ and $I$ be a homogeneous ideal {of $A$}. Then
	$${\Um}_n (A, I)=e_1\E_n(A, I) \ \text{for all n} \ \geq \max\{3,d+2\}.$$
\end{theorem}
\begin{proof}
	We may assume that the ring $R$ is reduced, see Lemma \ref{pnil}.
	We choose $\v \in {\Um}_n (A, I)$.  Let $\gamma(R)$ denote the number of minimal prime ideals of $R$. We prove the result by induction on $\gamma(R)$. For $\gamma(R) = 1$, the ring $R$ is a domain; therefore, the result follows by Lemma \ref{mainid}.
	

	Now we assume $\gamma(R) = m > 1$. Let $\p_1,\p_2,\dots,\p_m$ be minimal prime ideals of $R$. Since $R$ is a reduced ring, we have $\p_1\cap \ldots \cap \p_m = 0$. Let $\mathcal{J} = \p_1 \cap \dots \cap \p_{m-1}$ and $J = A \cap \mathcal{J} R[t, 1/ t]$. We have the inclusion of rings 
	$$R/\mathcal{J}\hookrightarrow A/ J \hookrightarrow R/\mathcal{J} [t, 1/ t].$$ We observe that $\gamma(R/ \mathcal{J}) = m - 1$. Let overlines denote the image modulo $J$. By induction hypothesis, there is a $\overline{\alpha}_1 \in \E_n(\overline{A}, \overline{I})$ such that $\overline{\v}~ \overline{\alpha}_1 = e_1$. Let $\alpha_1 \in \E_n(A, I)$ be a lift of $\overline{\alpha}_1$ and  $\w = \v \alpha_1$. Then 
	$\w\in  \Um_n(A, I \cap J)$ and $\v \sim_{ \E_n(A, I)} \w$.
	
	Let \ $\widetilde{}$ \  denote the image modulo $\q_m = A \cap \p_m R[t, 1/ t]$. We have inclusion of rings $$R/ \p_m \hookrightarrow \widetilde{A}\hookrightarrow R/ \p_m [t, 1/ t].$$
	We consider the image $\widetilde{\w}$ of $\w$ in $\Um_n(\widetilde{A}, \widetilde{I \cap J})$. Since $R/ \p_m$ is a domain, we have $\widetilde{\w}  \sim_{\E_n(\widetilde{A}, \widetilde{I \cap J})} e_1$  by Lemma \ref{mainid}.
	
	Lifting the elementary operations, we have $\w \sim_{\E_n(A, I \cap J)} \w'$ where $\w' \in  \Um_n(A, I \cap J \cap \q_m)$. Now we observe that $ J \cap \q_m = 0$ since $\p_1 \cap \ldots \cap \p_m = 0$. Therefore $\w' = e_1$ and consequently $\w \sim_{\E_n(A, I \cap J)} e_1$. Hence $\v \sim_{ \E_n(A, I)} e_1$ and the proof follows.
\end{proof}
\section{On transitive action of Elementary Transvection group $\ET(A \oplus P)$}
Let $R$ be a ring of dimension $d$. The aim of this section is to show that projective modules over graded $R$-subalgebras of $R[t, 1/t]$ of rank at least $d + 1$ have cancellation property and contain unimodular elements.

When one speaks of the cancellation property, the first question that comes to our mind is if the group of automorphisms on $A \oplus P$ acts transitively on $\Um(A \oplus P)$. 
We begin by proving the following preparatory lemma.

\begin{lemma} \label{L31}
	Let $k$ be a field and $A$ be a graded $k$-subalgebra of $k[t, 1/t]$ such that $t^m, t^{-n} \in A$ for some $ m, n \in \mathbb{N}$. Then $ A=k[t^l,t^{-l}]$ for some $l \in \NN$. 
\end{lemma}

\begin{proof} 
	We define $l=\min\{i\in \N \colon t^i\in A\}$ and $l^{\prime}=\min\{i\in \N \colon t^{-i}\in A\}$. Note that $t^{l}, t^{-l'}\in A$ and if $l^{\prime}<l $ then $t^{l-l^{\prime}} \in A$ which contradicts the minimality of $l$. So, $ l \leq l^{\prime}$ and similarly, by symmetry
	$ l^{\prime}\leq l$. Hence,  $l^{\prime}=l$ and $k[t^l, t^{-l}] \subseteq A$.
	
	Suppose $t^r \in A $ for some $r\in {\ZZ}$. By using the division algorithm, we have $r=lq+u$ such that either $u=0$ or $0<u<l$. If $u=0$ then $t^r \in k[t^l,t^{-l}]$ , otherwise for $0<u<l$, we have $t^u= t^{r-lq} \in A$ which is a contradiction to the minimality of $l$ . Therefore, $l$ divides $r$. Hence, $A \subseteq k[t^l, 1/t^l]$ and we are through.
\end{proof}

The method of proof of the next result is inspired by the work of Lindel and Wiemers, cf.  \cite[Theorem 2.6]{lindel1995unimodular}, \cite[Theorem 3.2]{wiemers1993cancellation}.  Their work also inspired \cite[Lemma 3.10]{dhorajia2010projective}.

\begin{lemma}\label{ord}
	Let  $R$ be an integral domain {of dimension $d$}, $A$ be a graded R-subalgebra of $R[t,1/t]$ and $I$ be an ideal of $A$ containing a homogeneous element. Let $P$ be a finitely generated projective $A$-module of rank  $r \geq \max\{d+1,2\}$. Then $\Um(A \oplus P, I)= e_1\ET(A \oplus P, I)$.
\end{lemma}
\begin{proof}
	
	We may assume that $A$ is Noetherian due to Lemma \ref{pidem}.
	We may further assume that $R \subsetneq A$, $I \neq 0$ and $r = d + 1$  since otherwise the result is obvious by Theorem \ref{bass}. 
	Let $S = R \setminus \{0\}$ and $k =  {S^{-1}R}$ denote the field of fractions of $R$.
	Then $S^{-1}A$  is a graded subring of $k[t, 1/ t]$. We choose a homogeneous element $u \in I$ of nonzero degree. By Lemma \ref{L31}, the localization $S^{-1}A_u$ is a PID. Consequently, the projective module $S^{-1}P_u$ over $S^{-1}A_u$ is free. Therefore, we obtain an element $s \in S$ such that $P_{su}$ is free. By Lemma \ref{P1}, we have $p _1, p_2, \ldots, p_r \in P$ and $ \phi_1, \phi_2, \ldots, \phi_r \in P^* = \Hom(P, A)$ and $m \in \NN$ such that 
	
	\begin{itemize}
		\item[1.] $(\phi_i(p_j))_{r \times r} = diagonal ((su)^m, (su)^m, \ldots, (su)^m)$.
		
		\item[2.]  $(su)^mP \subset F$ and $(su)^mP^* \subset G$ with $F = \sum_{i = 1}^r Ap_i$ and $G = \sum _{i = 1}^r A\phi_i$.
	\end{itemize}
	For notational convenience, we set $v = (su)^m$. 
	
	Let $(a, p) \in \Um(A \oplus P, I)$.
	We note that $\dim A/ (v^9) \leq d$. We use overlines to denote the quotient modulo the ideal $(v^9)$. Now $(\overline{a}, \overline{p}) \sim_{\ET(\overline{A} \oplus \overline{P}, \overline{I})} (1, 0)$, see Theorem \ref{bass}. Lifting the elementary operations, we obtain $$(a, p) \sim_{\ET(A \oplus P, I)} (a', p')\ \text{ for some} \  (a', p') \in \Um(A \oplus P, (v^9)).$$
	
	Now $p' \in v^9 P \subset v^8 F$, so there are $a_1, \ldots, a_r \in (v^8)$ such that 
	$p' = a_1 p_1 + \ldots + a_r p_r.$
	Since $(a', p') \in \Um(A \oplus P, (v^9))$, there are $a'' \in A$ and $\psi \in P^*$ such that $a'a'' + \psi(p') = 1$. This implies $a'a'' + a_1 \psi(p_1) + \ldots + a_r \psi(p_r) = 1$. Therefore, $$\v = (a', a_1, \ldots, a_r) \in {\Um}_{r + 1}(A , (v^8)).$$
	By Theorem \ref{main}, there is a matrix $\alpha \in \E_{r + 1} (A, (v^8))$ such that $\v \alpha = e_1$.  	
	
	For any ideal $I$ of a ring $R$, one has  $\E_n(R, I^4) \subset \E_n(I^2), n \geq 3$, cf.\cite[Lemma 2.7.(b)]{vasersteinSuslin1976serre}. 
	Therefore, we have $\alpha \in \E_{r + 1} (A, (v^8)) \subset \E_{r + 1} (v^4A)$. Due to the  commutator relations : $e_{ij}(xy) = [e_{i1}(x), e_{1j}(y)]$ for $i, j \neq 1$, we have the decomposition
	\[\alpha = \tau_1 \tau_2 \cdots \tau_k,\]
	where each $\tau_l$ equals an elementary matrix of type either $e_{1j}(xv^2)$ or  $e_{i1}(yv^2)$ for $x, y \in A$. Now we construct a sequence of transvections $\tilde{\tau_1}, \ldots, \tilde{\tau_k} \in \ET(v(A \oplus P)) \subset \ET(A \oplus P, I)$ as follows :
	\begin{enumerate}
		\item 
		If $\tau_l$ is the matrix $e_{i1}(yv^2)$ , then $(b_0, b_1, \ldots, b_{r})\tau_l =  (b_0 + yv^2 b_{i- 1}, b_1, \ldots, b_{r})$. We define $\tilde{\tau_l}$ as the  transvection given by $(b_0, q ) \xrightarrow{\tilde{\tau_l}} (b_0 + yv \phi_{i - 1}(q), q).$ We observe that 
		$$\tilde{\tau_l}(b_0, b_1p_1 + \ldots + b_r p_r) = (b_0 + yv^2 b_{i - 1},  b_1p_1 + \ldots + b_r p_r).$$
		
		\item
		If $\tau_l$ is the matrix $e_{1j}(xv^2)$, then $(b_0, b_1, \ldots, b_{r})\tau_l =  (b_0, b_1, \ldots, b_{j - 1} + xv^2b_0, \ldots, b_{r})$. We define $\tilde{\tau_l}$ as the  transvection given by $(b_0, q ) \xrightarrow{\tilde{\tau_l}} (b_0, q + xv^2b_0 p_{j -1}).$ We observe that 
		$$\tilde{\tau_l}(b_0, b_1p_1 + \ldots + b_r p_r) = (b_0,  b_1p_1 + \ldots +  (b_{j - 1} + xv^2b_0)p_{j -1} + \ldots + b_r p_r).$$
	\end{enumerate}
	
	Since $\v \alpha = e_1$, one clearly observes that $\tilde{\tau_k} \circ \cdots \circ \tilde{\tau_1}  (a', p') = (1, 0)$ which implies 
	$$(a', p')\sim_{\ET(A \oplus P, I)} e_1.$$ Hence the proof is completed. 
\end{proof}

The theorem below can be proved by using Lemma \ref{ord} above and by induction on the number of minimal prime ideals following the method of the proof of Theorem \ref{main} verbatim. We skip the details.

\begin{theorem}{\label{43}}
	Let  $R$ be a ring {of dimension d}, $A$ be a graded R-subalgebra of $R[t,1/t]$ and $I$ be a homogeneous ideal of $A$. Let $P$ be a finitely generated projective $A$-module of rank  $r \geq \max\{d+1,2\}$. Then $\Um(A \oplus P, I)= e_1\ET(A \oplus P, I)$.
\end{theorem}

A projective module of rank $1$ over a ring is always cancellative, cf. \cite [Chapter I, Proposition 6.6]{lam2006serre}. The cancellation property of a projective module $P$ of rank at least $2$ over a ring $A$ ($Q \oplus A \cong P\oplus A$ $\implies$ $Q \cong P$) is tantamount to transitive action of  $\Aut(A\oplus P)$ on $\Um(A \oplus P)$, cf. \cite[Corollary 3.5]{BassH}. Therefore, the previous theorem yields the following:

\begin{corollary}\label{cancel}
	Let $R$ be a ring of dimension $d$ and $A$ be a graded $R$-subalgebra of $R[t,1/t]$. Let $P$ and $Q$ be projective $A$-modules of rank $r\geq d + 1$ satisfying $Q \oplus A \cong P\oplus A$. Then $Q \cong P$.  
\end{corollary}

We recall the definition of Isotopy subgroups, cf. \cite[\S2, 2.2]{bhatwadekar1982stability}.

\begin{definition}{\rm(Isotopy Subgroups)}
	Let $I$ be an ideal of a ring $R$ and  {$n\geq 2$}. We define 
	\[{\h}_n(R, I) = \{M \in \GL_n(R, I) : \text{there is an $\alpha(x) \in \GL_n(R[x], IxR[x])$ such that $\alpha(1) = M$}\}.\]
	It is easy to check that $\h_n(R, I)$ is a normal subgroup of $\GL_n(R, I)$. The group $\h_n(R, I)$ is called the isotopy subgroup of $\GL_n(R, I)$. If $I = R$, then $\h_n(R, I)$ is abbreviated as $\h_n(R)$. If $R$ is a reduced ring, then one easily observes $\h_n(R, I) \subset \SL_n(R, I)$.
\end{definition}	

We state some important facts about isotopy groups.
Let $I$ be an ideal of a ring $A$.
\begin{Fact}
	
	The group $\E_n(A, I)$ is a subgroup of $\h_n(A, I)$. Let $M \in \E_n(A, I)$. If $M = \prod_{l = 1}^m \gamma_l^{-1} e_{i_l j_l}(a_l) \gamma_l$  for  $\gamma_l \in \GL_n(A)$, $a_l \in I$, then we define $\alpha(x) = \prod_{l = 1}^m \gamma_l^{-1} e_{i_l j_l}(a_l x) \gamma_l$. Clearly, $\alpha(x) \in \SL_n(R[x], IxR[x])$ such that $\alpha(1) = M$, so $M \in \h_n(R, I)$.
\end{Fact}

\begin{Fact}\label{F1}
	
	If $A$ is a graded $k$-subalgebra of $k[t, 1/ t]$ where $k$ is a field, then $SL_n(A) = \h_n(A)$. Let $M \in \SL_n(A)$. We show that $M \in \h_n(A)$ considering several cases. 
	Suppose $A \subset k[t]$ and consider $M = (a_{ij}(t))_{n \times n}$ for $a_{ij}(t) \in k[t]$. We define $\alpha(x) = (a_{ij}(tx))_{n \times n} \in \SL_n(A[x])$. We have $\alpha(1) = M$ and $\alpha(0) \in \SL_n(k) = \E_n(k)$. We choose $\beta(x) \in \E_n(k[x])$ such that $\beta(1) = \alpha(0)$ and $\beta(0) = I_n$. Then $\gamma(x) = \alpha(x)[\beta(1 - x)]^{-1} \in \SL_n(A[x], (x))$ such that $\gamma(1) = M$. Therefore, $M \in \h_n(A)$. If $A \subset k[1/ t]$, then one  similarly shows that $M \in \h_n(A)$. If $A \not\subset k[t] \cup k[1 / t]$, then $A$ is a Euclidean domain by Lemma \ref{L31}. In this case we have $SL_n(A) = \E_n(A) = \h_n(A)$, so $M \in \h_n(A)$. Therefore, we get $\SL_n(A) \subset \h_n(A)$. Since the reverse inclusion is obvious, our claim follows.
\end{Fact}

\begin{Fact}\label{F2}
	If $R$ is an  Artinian ring and $A$ is a graded $R$-subalgebra of $R[t, 1/ t]$, then  $\SL_n(A) \subset \h_n(A)$. Let $M \in \SL_n(A)$. Let overline denote the quotient modulo the nil radical. 
	Note that $\overline{R} = R/ Nil(R)$ is a finite product of fields, so $\h_n(\overline{A}) = \SL_n(\overline{A})$ by the previous  fact. Therefore, we obtain $\overline{\alpha}(x) \in \SL_n(\overline{A}[x], (x))$ such that $\overline{\alpha}(1) = \overline{M}$. Since units modulo a nilpotent ideal can be lifted to units, we get $\alpha(x) \in \GL_n(A[x], (x))$ which is a preimage of $\overline{\alpha}(x)$. Then $P = \alpha(1) -  M$ is a matrix whose all entries are nilpotent. We observe that $\alpha(1) = M + P = M(I_n + N)$ where $N = M^{-1} P$ is a nilpotent matrix as all the entries of $N$ are nilpotent elements. We define $\beta(x) = \alpha(x) (I_n + x N)^{-1}$, then we have $\beta(x) \in \GL_n(A[x], (x))$ and $\beta(1) = M$. It follows that $M \in \h_n(A)$, consequently $SL_n(A) \subset \h_n(A)$.
\end{Fact}

%
Before stating the next lemma, we set the following convention:
if $f : R \rightarrow S$ is a ring homomorphism, we denote the induced maps $R[x] \rightarrow S[x]$, $\sum_{i = 0}^na_ix^n \mapsto  \sum_{i = 0}^nf(a_i)x^n$;  $\Um_n(R) \rightarrow \Um_n(S)$, $(a_i)_{1 \times n} \mapsto (f(a_i))_{1 \times n}$ and $\GL_n(R) \rightarrow \GL_n(S)$, $(a_{ij})_{n \times n} \mapsto (f(a_{ij}))_{n \times n}$  all by $f$ for the sake of notational convenience.

\begin{lemma}\label{F4}
	{If} $f: R \rightarrow R/I$ is a retract map, then ${\h}_n(R, I)= {\h}_n(R) \cap \SL_n(R,I)$.
\end{lemma}
\begin{proof}	
	We choose $g: R/I \rightarrow R$ such that $f \circ g = {Id_{R/I}}$.
	Let $\beta \in {\h}_n(R) \cap \SL_n(R,I)$. We choose $F(x) \in \GL_n(R[x], (x))$  such that $F(1) = \beta$. Let $\widetilde{F}(x) = F(x) [g \circ f(F(x))]^{-1}$. It is clear that $\widetilde{F}(x) \in \GL_n(R[x], (x))$. 
	We observe that $f(\widetilde{F}(x)) = I_n$, which implies  $\widetilde{F}(x) \in \GL_n(R[x], IR[x])$. Therefore, $\widetilde{F}(x) \in \GL_n(R[x], IxR[x])$. Note that
	$$\widetilde{F}(1) = F(1) [g \circ f(F(1))]^{-1} = \beta [g \circ f(\beta)]^{-1} = \beta [g(I_n)]^{-1}= \beta.$$ 
	Therefore, $\beta \in \h_n(R, I)$ and consequently ${\h}_n(R) \cap \SL_n(R,I) \subset {\h}_n(R, I)$. As the reverse inclusion is clear, the proof follows.
\end{proof}

\begin{lemma} \label{uni1}
	Let $I$ be an ideal of a ring $R$ such that the quotient map $f: R \rightarrow R/ I$ is a retract map.
	Assume $\Um_n(R)= e_1\h_n(R), n \geq 2$. If $\v_1, \v_2 \in \Um_n(R)$ are such that $\v_1 \equiv \v_2 \vpmod I$, then $\v_1 \sim_{\h_n(R, I)} \v_2$.
\end{lemma}

\begin{proof}
	We choose a ring homomorphism $g: R/I \rightarrow R$  such that $ f\circ g = Id_{R/I}$.
	Let $$\u = f(\v_1) = f(\v_2) \in {\Um}_n(R/ I)\ \text{and} \  \v = g(\u) \in  {\Um}_n(R).$$ 
	If we show that $\v_1 \sim_{\h_n(R, I)} \v$ and $\v_2 \sim_{\h_n(R, I)} \v$, then we are through. We only show that $\v_1 \sim_{\h_n(R, I)} \v$ as the other one follows similarly.
	
	Both  $\v, \v_1$ are in $\Um_n(R)$. By the given hypothesis, there is an $\alpha \in \h_n(R)$ such that $\v_1 = \v \alpha$. Applying $g \circ f$ on both sides, we get $\v = \v [g \circ f(\alpha)]$. 
	Let $\beta = [g \circ f(\alpha)]^{-1} \alpha$.
	We observe
	\[f(\beta) = f([g \circ f(\alpha)]^{-1} \alpha) = [f \circ g \circ f(\alpha)]^{-1} [f(\alpha)] = [f(\alpha)]^{-1} [f(\alpha)] = I_n.
	\]
	This implies $\beta \in \SL_n(R, I)$. Since $\alpha \in \h_n(R)$, we have $g \circ f(\alpha) \in \h_n(R)$. Consequently we obtain $\beta = [g \circ f(\alpha)]^{-1} \alpha \in \h_n(R) \cap \SL_n(R, I)$.
	By Lemma \ref{F4}, we have $\beta \in \h_n(R, I)$. An easy computation yields  $\v_1 = \v \beta$. This implies $\v_1 \sim_{\h_n(R, I)} \v$ and hence the proof is completed here.
\end{proof}

\begin{proposition}\label{uni2}
	Let $I$ be an ideal of a ring $R$ such that the quotient map $f: R \rightarrow R/ I$ is a retract map. Let $P$ be a projective $R$-module of rank $n \geq 2$. Assume $s \in R$ and $T$ is a multiplicatively closed subset of $R$ such that 
	\begin{enumerate}
		\item
		for any $t \in T$, we have $R=sR+tR$. 
		\item 
		the localization $P_{sT}$ is a free $R_{sT}$ module,
		\item
		$\Um_n(R_{sT})= e_1\h_n(R_{sT})$. 
	\end{enumerate}
	Then the natural map $\Um(P)\rightarrow \Um(P/IP)$ is surjective if both the natural maps $\Um(P)_s\rightarrow \Um(P/IP)_s$ and $\Um(P_{T})\rightarrow \Um(P_T/IP_T)$ are surjective.
\end{proposition}

\begin{proof}
	Let $\overline p \in \Um(P/IP)$ where $p \in P$. Consider the following patching diagram,
	$$\xymatrix{
		& P \ar[d] \ar[r]
		& P_s \ar[d] \\
		& P_{T} \ar[r] & P_{sT} }$$
	
	Under given hypothesis, we can find $q_1 \in \Um (P_s)$ and $q_2 \in  \Um (P_{T})$ such that $ q_{1} \equiv p_{s} \pmod{IP_{s}}$ and  $q_{2} \equiv p_{T}\pmod{IP_{T}}$. Note that 
	\[(q_1)_{T}, (q_2)_{s} \in \Um(P_{sT}) = {\Um}_n(R_{sT})\ \text{and} \ (q_1)_{T} \equiv (q_2)_{s} \equiv{p_{sT}} \vpmod {IP_{sT}}.\]
	By  Lemma \ref{uni1}, there exists $ \tau \in \h_{n}(R_{sT}, I)$ such that $\tau(q_1)_{T}= (q_{2})_{s}$.
	We know that $P_{T} = {\small \underset{t \in T}{\varinjlim}} P_t$. We may choose $t \in T$ such that 
	$q_2 \in  \Um (P_{t})$ satisfying  $q_{2} \equiv p_{t}\pmod{IP_{t}}$ and $ \tau \in \h_{n}(R_{st}, I)$ satisfying $\tau(q_1)_{t}= (q_{2})_{s}$.

	Now $\tau = (\tau_2)_s (\tau_1)_ {t}$ for some $\tau_1 \in \Aut(P_{s}, IP_{s})$ and $\tau_2 \in \Aut(P_{t}, IP_{t})$, cf. \cite[\S 2]{plumstead1983conjectures}. Then we have $\tau_1(q_1)_t = \tau_2^{-1}(q_2)_s$. Patching these elements, we obtain $q \in \Um(P)$ such that $q_t = \tau_2^{-1}(q_2)$ and $q_s = \tau_1(q_1)$. We observe that $q_t \equiv q_2 \equiv p_t \vpmod I$ and $q_s \equiv q_1 \equiv p_s \vpmod I$. Hence $q \equiv p \vpmod I$ and the proof follows.
\end{proof}

If $ {A}$ is a Noetherian ring of dimension $1$ and $P$ is a projective ${A}$-module of rank $r$, then $P \cong {A}^{r - 1} \oplus \det P$, where $\det P = \wedge^r P$ , cf. \cite[Theorem 1]{serre1958modules}. This fact is crucial to our next result.

\begin{theorem}\label{uni3}
	Let $R$ be a ring of dimension $d$ and $A$ be a graded $R$-subalgebra of $R[t, 1/ t]$. Let $P$ be a projective $A$-module of rank $r \geq d + 1$ such that $\det P$ is extended from $R$.  
	Let $\phi: A \rightarrow R$ be an $R$-algebra homomorphism and $I= Ker(\phi)$. Then the natural map $\pi: \Um(P) \rightarrow \Um(P/IP)$ is surjective.\end{theorem}  

\begin{proof} 
	We may assume that $A$ is Noetherian due to Lemma \ref{pidem}.
	We begin with the observation that $\pi$ is surjective whenever $P$ is free.
	We may assume that $R$ is reduced. If $\dim R = 0$, then $R$ is a product of fields and $\dim A \leq 1$. We have $P \cong \det P \oplus A^{r-1}$. Note that $\det P$ is a free $A$-module since it is extended from $R$. Therefore, $P$ is free and the result follows.
	
	We assume that $\dim R  \geq 1$ and $r \geq 2$. Let $S$ denote the set of nonzero divisors of $R$. We observe that $\dim S^{-1}R =0$ and $\det S^{-1}P$ is extended from $S^{-1}R$.
	By the argument in the last paragraph, $S^{-1}P$ is a free $S^{-1}A$ module.
	Since $P$ is finitely generated, we get $s \in S$ such that $P_s$ is a free $ A_s$-module, so the natural map $\Um(P_s) \rightarrow \Um(P_s/IP_s)$ is surjective.
	
	Let $T = 1 + sR$. Note that $\spec A_{T} = V(s) \sqcup D(s)$. Since $s$ is a nonzero divisor, we have  $\dim V(s) \leq d$. We observe $\dim R_{sT} \leq d - 1$ and ${R_{sT}} \subset A_{sT}  \subset R_{sT}[t, 1/ t]$, so $\dim D(s) \leq d$.
	
	Let $\overline p \in  \Um (P/IP)_T$, where $p \in P_T$. Then we have $a \in IA_T$ such that $(a, p) \in \Um(A_T \oplus P_T)$. By a standard argument using generalised dimension functions, cf.  \cite[Theorem A]{eisenbud1973generating}, \cite[p. 1420]{plumstead1983conjectures}, there exists $q \in P_T$ such that $p + aq \in \Um(P_T)$. We observe that $p + aq \equiv p  \vpmod {IP_T}$. Therefore, the natural map $\Um(P_{T})\rightarrow \Um(P_T/IP_{T})$ is a surjection. 
	
	Now $\dim A_{sT} \leq d \leq r - 1$. If $r \geq 3$, then we have  $\Um_r(A_{sT})=e_1 \E_r(A_{sT}) = e_1\h_r(A_{sT})$ by Theorem \ref{main}. If $r = 2$, then $d \leq 1$ and $\dim R_{sT} = 0$, therefore  $\Um_2(A_{sT}) = e_1\SL_2(A_{sT}) = e_1\h_2(A_{sT})$ by Fact \ref{F2}. Hence the result follows by Proposition \ref{uni2} as $\phi$ is a retract map.
\end{proof}

Any graded $R$-subalgebra $A$ of $R[t, 1/ t]$ admits an $R$-algebra homomorphism $\phi : A \rightarrow R$ given by  $t \mapsto 1$. 
Therefore, the next corollary follows from Theorem \ref{uni3} and \cite[Theorem 1]{serre1958modules}.

\begin{corollary}\label{eum}
	Let $R$ be a ring of dimension $d$ and $A$ be a graded $R$-subalgebra of $R[t, 1/ t]$. Let $P$ be a projective $A$-module of rank $r \geq d + 1$ such that its determinant is extended from $R$. Then $P$ has a unimodular element.
\end{corollary}

In conclusion, we state the following example due to Bhatwadekar et al., cf. \cite{bhatwadekar2017some}.
%

\begin{example}\label{bhatex}
	Let $F = Z^3 - T^2 \in k[Z, T]$, then $\frac{k[Z, T]}{(F)} = k[W^2, W^3]$. Let $ \phi : k[X_1, Y, Z, T] \rightarrow k[X_1, \ldots, X_4]$ be a $k[X_1]$-algebra homomorphism defined by $\phi(Z) = X_2^2 - X_1X_3$, $\phi(T) = X_2^3 - X_1X_4$, $\phi(Y) = \frac{1}{X_1} F(\phi(Z), \phi(T)) = - 3X_2^4 X_3 + 3X_1X_2^2 X_3^2  - X_1^2 X_3^3  + 2X_2^3 X_4 - X_1X_4^2$. By \cite[Theorem 5.4]{bhatwadekar2017some}, $\phi$ induces an isomorphism between $A = \frac{k[X_1, Y, Z, T]}{(X_1Y - F)}$ and $B = k[X_1, \phi(Y), \phi(Z), \phi(T)]$. Let $\n = B \cap \m$ for $\m = (X_1,\ldots, X_4)$. Then $\mu(\n) = \mu(\overline{X_1}, \overline{Y}, \overline{Z}, \overline{T}) = 4$. 
	If $B$ were a graded subalgebra of $k[X_1, \ldots, X_4]$, $\deg(X_i) = 1$, we would have $B = k[X_1, X_2^2 - X_1X_3, X_1X_4,  X_2^3, X_1X_4^2,  X_2^3 X_4,  - 3X_2^4 X_3 + 3X_1X_2^2 X_3^2  - X_1^2 X_3^3]$ 
	which is absurd as $\mu(\n) = 7$ in this case. However $A$ is graded with $\deg(X_1) = \deg(Y) = \deg(T) = 3$ and $\deg(Z) = 2$.
	By \cite[Theorem 5.4]{bhatwadekar2017some},  $A (\cong B)$ is a UFD, $Pic(A) = Pic(B) = 0$ and $K_0(A)  = K_0(B) \not= \ZZ$. Therefore, both $A$ and $B$ admit non-free projective modules of trivial determinant.
\end{example}
%

\begin{acknowledgement}
	This article is a part of the first author's PhD thesis. She  is grateful to the National Board of Higher Mathematics for providing financial support, file number 0203/21/2019-RD-II/13102. The second author is thankful to INSPIRE Faculty Fellowship, DST, reference no DST/INSPIRE/04/2018/000522.
\end{acknowledgement}

\bibliographystyle{acm}
\bibliography{reference.bib}

\end{document}